\documentclass[12pt,a4paper]{amsart}
\usepackage{mathrsfs}
\usepackage{amssymb,amsmath,amsthm,color}
\usepackage{graphicx,mcite}
\usepackage{hyperref}
\usepackage{url}
\usepackage{setspace}
\usepackage{enumerate}
\newtheorem{theorem}{Theorem}[section]
\newtheorem{lemma}[theorem]{Lemma}
\newtheorem{proposition}[theorem]{Proposition}

\theoremstyle{definition}
\newtheorem{definition}[theorem]{Definition}

\newtheorem{remark}[theorem]{Remark}

\numberwithin{equation}{section}

\begin{document}

\title[Heisenberg uniqueness pair]
{Heisenberg uniqueness pairs for some algebraic curves in the plane}
%\runningtitle{Sets of Injectivity for Weighted Twisted Spherical Means}

\author{Deb Kumar Giri and R. K. Srivastava}
 %Address of record for the research reported here
%\address{Deb Kumar Giri, Department of Mathematics, Indian Institute of Technology, Guwahati, India 781039.}

\address{Department of Mathematics, Indian Institute of Technology, Guwahati, India 781039.}
%\curraddr{Department of Mathematics and Statistics,
%Case Western Reserve University, Cleveland, Ohio 43403}
\email{deb.giri@iitg.ernet.in, rksri@iitg.ernet.in}
%\thanks will become a 1st page footnote.

%    General info

\date{\today}


\begin{abstract}
A Heisenberg uniqueness pair is a pair $\left(\Gamma, \Lambda\right)$, where $\Gamma$
is a curve and $\Lambda$ is a set in $\mathbb R^2$ such that whenever a finite Borel
measure $\mu$ having support on $\Gamma$ which is absolutely continuous with respect to
the arc length on $\Gamma$ satisfies $\hat\mu\vert_\Lambda=0,$ then it is identically $0.$
In this article, we investigate the Heisenberg uniqueness pairs corresponding to the spiral,
hyperbola, circle and certain exponential curves. Further, we work out a characterization of
the Heisenberg uniqueness pairs corresponding to four parallel lines. In the latter case,
we observe a phenomenon of interlacing of three trigonometric polynomials.
\end{abstract}

\maketitle

\section{Introduction}\label{section1}
The concept of the Heisenberg uniqueness pair has been first introduced in an
influential article by Hedenmalm and Montes-Rodríguez (see \cite{HR}). We would
like to mention that Heisenberg uniqueness pair up to a certain extent is similar
to an annihilating pair of Borel measurable sets of positive measure as described by
Havin and Joricke \cite{HJ}. Further, the notion of Heisenberg uniqueness pair has
a sharp contrast to the known results about determining sets for measures by
Sitaram et al. \cite{BS, RS}, due to the fact that the determining set $\Lambda$ for
the function $\hat\mu$ has also been considered a thin set.

\smallskip

In addition, the question of determining the Heisenberg uniqueness pair for a
class of finite measures has also a significant similarity with the celebrated
result due to M. Benedicks (see\cite{B}). That is, support of a function $f\in L^1(\mathbb R^n)$
and its Fourier transform $\hat f$ cannot be of finite measure simultaneously.
Later, various analogues of the Benedicks theorem have been investigated in
different set ups, including the Heisenberg group and Euclidean motion groups
(see \cite{NR, PS, SST}).

\smallskip

In particular, if $\Gamma$ is compact, then $\hat\mu$ is a real analytic function
having exponential growth and it can vanish on a very delicate set. Hence
in this case, finding the Heisenberg uniqueness pairs becomes little easier.
However, this question becomes immensely difficult when the measure is supported
on a non-compact curve. Eventually, the Heisenberg uniqueness pair is a natural
invariant to the theme of the well studies uncertainty principle for the Fourier
transform.

\smallskip

In the article \cite{HR}, Hedenmalm and Montes-Rodríguez have shown that the pair
(hyperbola, some discrete set) is a Heisenberg uniqueness pair. As a dual problem,
a weak$^\ast$ dense subspace of $L^{\infty}(\mathbb R)$  has been constructed to
solve the Klein-Gordon equation. Further, a complete characterization of the
Heisenberg uniqueness pairs corresponding to any two parallel lines has been
given by Hedenmalm and Montes-Rodríguez (see \cite{HR}).

\smallskip

Afterward, a considerable amount of work has been done pertaining to the
Heisenberg uniqueness pair in the plane as well as in the higher dimensional
Euclidean spaces.

\smallskip

Recently, N. Lev \cite{L} and P. Sjolin \cite{S1} have independently shown that
circle and certain system of lines are HUP corresponding to the unit circle $S^1.$
Further, F. J. Gonzalez Vieli \cite{V1} has generalized HUP corresponding to circle
in the higher dimension and shown that a sphere whose radius does not lie in the zero
set of the Bessel functions $J_{(n+2k-2)/2};~k\in\mathbb Z_+,$ the set of non-negative
integers, is a HUP corresponding to the unit sphere $S^{n-1}.$

\smallskip

Per Sjolin \cite{S2} has investigated some of the Heisenberg uniqueness pairs
corresponding to the parabola. Subsequently, D. Blasi Babot \cite{Ba} has given a
characterization of the Heisenberg uniqueness pairs corresponding to a certain
system of three parallel lines. However, an exact analogue for the finitely many
parallel lines is still open.

\smallskip

 In a major development, P. Jaming and K. Kellay \cite{JK} have given a unifying proof
for some of the Heisenberg uniqueness pairs corresponding to the hyperbola, polygon,
ellipse and graph of the functions $\varphi(t)=|t|^\alpha,$ whenever $\alpha>0,$ via
dynamical system approach.

\smallskip

Let $\Gamma$ be a finite disjoint union of smooth curves in $\mathbb R^2.$
Let $X(\Gamma)$ be the space of all finite complex-valued Borel measure
$\mu$ in $\mathbb R^2$ which is supported on $\Gamma$ and absolutely
continuous with respect to the arc length measure on $\Gamma$. For
$(\xi,\eta)\in\mathbb R^2,$ the Fourier transform of $\mu$ is defined by
\[\hat\mu{( \xi,\eta)}=\int_\Gamma e^{-i\pi(x\cdot\xi+ y\cdot\eta)}d\mu(x,y).\]
In the above context, the function $\hat\mu$ becomes a uniformly continuous bounded
function on $\mathbb R^2.$ Thus, we can analyze the pointwise vanishing nature of the
function $\hat\mu.$
\begin{definition}
Let $\Lambda$ be a set in $\mathbb R^2.$ The pair $\left(\Gamma, \Lambda\right)$
is called a Heisenberg uniqueness pair for $X(\Gamma)$ if any $\mu\in X(\Gamma)$
satisfying $\hat\mu\vert_\Lambda=0,$ implies $\mu=0.$

\end{definition}
Since the Fourier transform is invariant under translation and rotation, one can
easily deduce the following invariance properties about the Heisenberg uniqueness
pair.
\smallskip

\begin{enumerate}[(i)]
  \item Let $u_o, v_o\in\mathbb R^2.$ Then the pair $(\Gamma,\Lambda)$
is a HUP if and only if the pair $(\Gamma+ u_o,\Lambda+v_o)$ is a HUP.

\smallskip

\item Let $T : \mathbb R^2\rightarrow \mathbb R^2$ be an invertible
linear transform whose adjoint is denoted by $T^\ast.$  Then  $(\Gamma,\Lambda)$
is a HUP if and only if $\left(T^{-1}\Gamma,T^\ast\Lambda\right)$  is a HUP.
\end{enumerate}
Now, we state first known results on the Heisenberg uniqueness pair due to
Hedenmalm and Montes-Rodríguez \cite{HR}. After that, we briefly indicate the
progress on this recent problem.

\begin{theorem}\cite{HR}\label{th9}
Let $\Gamma=L_1\cup L_2,$  where $L_j;~j=1,2$ are any two parallel straight lines
and $\Lambda$ a subset of $\mathbb R^2$ such that $\overline{\pi(\Lambda)}=\mathbb R.$
Then $(\Gamma,\Lambda)$ is a Heisenberg uniqueness pair if and only if the set
\begin{equation}\label{exp24}
\widetilde\Lambda={\pi_1^a(\Lambda)}\cup\left[{{\pi_1^b(\Lambda)}\smallsetminus{\pi_1^c(\Lambda)}}\right]
\end{equation}
is dense in $\mathbb R$.
\end{theorem}
Here we avoid to mention the notations appeared in (\ref{exp24}) as they are bit involved,
however, we have written down the same notations as in the article \cite{HR}. Though, their
main features can be perceived in Section \ref{section3}.
\begin{theorem}\cite{HR}\label{th18}
Let $\Gamma$ be the hyperbola $x_{1}x_{2}=1$ and $\Lambda_{\alpha,\beta}$ a lattice-cross
defined by
\[\Lambda_{\alpha,\beta}=\left(\alpha\mathbb Z\times\{0\}\right)\cup\left(\{0\}\times\beta\mathbb Z\right),\]
where $\alpha, \beta$ are positive reals. Then $\left(\Gamma,\Lambda_{\alpha,\beta}\right)$ is a Heisenberg
uniqueness pair if and only if $\alpha\beta\leq1$.
\end{theorem}

For $\xi\in\Lambda,$ define a function $e_\xi$ on $\Gamma$ by $e_{\xi}(x)=e^{i\pi x\cdot\xi}.$
As a dual problem to Theorem \ref{th18}, Hedenmalm and Montes-Rodríguez \cite{HR} have proved
the following density result which in turn solve the one-dimensional Kein-Gordon equation.
\begin{theorem}\cite{HR}\label{th10}
The pair  $(\Gamma,\Lambda)$ is a Heisenberg uniqueness pair  if and only if
the set $\{e_{\xi}:~\xi\in\Lambda\}$ is a weak$^\ast$ dense subspace of
$L^{\infty}(\Gamma).$
\end{theorem}

\begin{remark}\label{rk6}
In particular, for $\Gamma$ to be an algebraic curve, the question of Heisenberg uniqueness
pair can be understood through a partial differential equation (PDE). That is, if $\Gamma$
is the zero set of a polynomial $P$ on $\mathbb R^2,$ then $\hat\mu$ satisfies the PDE
\[P\left(\dfrac{\partial_1}{i\pi}, \dfrac{\partial_2}{i\pi}\right)\hat\mu=0\]
with initial condition $\hat\mu\vert_\Lambda=0.$ This formulation may help potentially
in determining the geometrical structure of the set $Z(\hat\mu),$ the zero set of the
function $\hat\mu.$ If we consider $\Lambda$ to be contained in $Z(\hat\mu),$ then
$\left(\Gamma, \Lambda\right)$ is not a HUP. Hence the question of the HUP arises
when $\Lambda$ has located away from $Z(\hat\mu).$
\end{remark}
\smallskip

In the case when $\mu$ is supported on a circle, the function  $\hat\mu$
becomes real analytic and hence it could vanish at most on a very thin set. Thus,
there are an enormous number of candidates for $\Lambda$ such that $\left(\Gamma,\Lambda\right)$
is a HUP. Some of the Heisenberg uniqueness pairs corresponding to circle has been
independently investigated by N. Lev and P. Sjolin. Following are their
main results. For more details, we refer to \cite{L, S1}.

\begin{theorem}\cite{L, S1}\label{th16}
Let $\Gamma=S^1$ be the unit circle.
\smallskip

\noindent $(i)$ Let $\Lambda$ be a circle of radius $r.$ Then $\left(\Gamma,\Lambda\right)$
is a HUP if and only if $J_{k}(r)\neq0$ for all $k\in\mathbb Z_+.$
\smallskip

\noindent $(ii)$ Let $\Lambda$ be a straight line. Then $\left(\Gamma,\Lambda\right)$
is not a HUP.
\smallskip

\noindent $(iii)$ Let $\Lambda=L_1\cup L_2,$ where $L_j;~j=1,2$ are two straight lines.
If $L_{1}$ and $L_{2}$ are parallel, then $\left(\Gamma,\Lambda\right)$ is a HUP.
\smallskip

\noindent $(iv)$ Let $L_j;~j=1,2,\ldots, N$ be the $N$ different straight
lines which intersect at one point and angle between any of two lines out
of these $N$ lines is of the form $\pi\alpha.$ Let $\Lambda_N=\bigcup\limits_{j=1}^N L_j.$
Then  $\left(\Gamma,\Lambda_N\right)$ is not a HUP if and only if $\alpha$ is rational.
\end{theorem}
In contrast to the case of finitely many straight lines, P. Sjolin \cite{S1} has
shown that if $\Lambda=\bigcup\limits_{k=1}^{\infty}L_k, $ where $\{L_k\}$ is a
sequence of straight lines which intersect at one point. Then $(S^1,\Lambda)$
is a HUP.

\begin{remark}\label{rk5}
Since we know that any homogeneous harmonic polynomial on $\mathbb R^2$ can be
expressed as $Ar^j\sin(j\theta+\delta)$ for some $j\in\mathbb N$ and
$\delta\in[0, 2\pi)$ (see \cite{FNS}), up to some rotation and translation,
we can think of $\Lambda_N=\bigcup\limits_{k=1}^N L_k,$ appeared in Theorem
\ref{th16} $(iv),$ as the zero set of some homogeneous harmonic polynomial.
If $(S^1,\Lambda)$ is a Heisenberg uniqueness pair, then the set $\Lambda$
must be away from the zero set of any homogeneous harmonic polynomial.
However, the converse is not true. Since $(S^1,\Lambda)$ is not a HUP if
$\Lambda$ is a circle whose radius lie in the zero set of some Bessel function.
Thus, it is an interesting question to examine the exceptional sets for the
Heisenberg uniqueness pairs corresponding to circle.
\end{remark}

Subsequently, some of the Heisenberg uniqueness pairs corresponding to
the parabola have been obtained by P. Sjolin \cite{S2}. Let $|E|$ denotes
the Lebesgue measure of the set $E\subset\mathbb R.$

\begin{theorem}\cite{S2}\label{th11}
Let $\Gamma$ denote the parabola $y= x^2.$
\smallskip

\noindent $(i)$ Let $\Lambda=L$ be a straight line. Then $\left(\Gamma,\Lambda\right)$
is a HUP if and only if $L$ is parallel to the X-axis.
\smallskip

\noindent $(ii)$ Let $\Lambda=L_{1}\cup L_{2},$ where $L_j;~j=1,2$ are two
different straight lines. Then $\left(\Gamma,\Lambda\right)$ is a HUP.
\smallskip

\noindent $(iii)$ Let $L_j;~j=1,2$ be two different straight lines which are
not parallel to the $X$-axis. Let $E_j\subset L_j$ and $|E_j|>0;~j=1,2.$
If $\Lambda=E_{1}\cup E_{2},$ then $\left(\Gamma,\Lambda\right)$ is a HUP.
\end{theorem}

The question of Heisenberg uniqueness pair in the higher dimension
has been first taken up by F. J. Gonzalez Vieli \cite{V1, V2}.
\begin{theorem}\cite{V1}\label{th14}
Let $\Gamma=S^{n-1}$ be the unit sphere in $\mathbb R^n$ and $\Lambda$
a sphere of radius $r.$ Then $\left(\Gamma,\Lambda\right)$ is a HUP
if and only if $J_{(n+2k-2)/2}(r)\neq0 $ for all $k\in\mathbb Z_+.$
\end{theorem}

\begin{theorem}\cite{V2}\label{th15}
Let $\Gamma$ be the paraboloid $x_{n}=x_{1}^2+x_{2}^2+\cdots+x_{n-1}^2$
in $\mathbb R^n$ and $\Lambda$ an affine hyperplane in $\mathbb R^n$
of dimension $~n-1.$ Then $\left(\Gamma,\Lambda\right)$ is a HUP if and only
if $\Lambda$ is parallel to the hyperplane $x_{n}=0.$
\end{theorem}

Let $\Gamma$ denote a system of three parallel lines in the plane that can be
expressed as $\Gamma=\mathbb R\times\{\alpha,\beta,\gamma\},$ where
$\alpha<\beta<\gamma$ and $(\gamma-\alpha)/(\beta-\alpha)\in\mathbb N.$
By the invariance properties of HUP, one can assume that $\Gamma=\mathbb R\times \{0,1,p\},$
for some $p\in\mathbb N$ with $p\geq2.$ The following characterization for
the Heisenberg uniqueness pairs corresponding to the above mentioned three
parallel lines has been given by D. B. Babot \cite{Ba}.

\begin{theorem}\cite{Ba}\label{th3}
Let $\Gamma= \mathbb R\times \{0,1,p\},$ for some $p\in\mathbb{N}$ with $p\geq2$
and $\Lambda\subset\mathbb R^2$ a closed set which is $2$-periodic with respect
to the second variable. Then $\left(\Gamma,\Lambda\right)$ is a HUP if and only
if the set
\begin{equation}\label{exp23}
\widetilde\Lambda= {\Pi^{3}(\Lambda)}\cup\left[{{\Pi^{2}(\Lambda)}\smallsetminus{\Pi^{2^\ast}(\Lambda)}}\right]
\cup\left[{{\Pi^{1}(\Lambda)}\smallsetminus{\Pi^{1^\ast}(\Lambda)}}\right]
\end{equation}
is dense in $\mathbb R$.
\end{theorem}
For the notations appeared in Equation (\ref{exp23}), we would like to refer the article \cite{Ba},
as those notations are quite involved. However, the nature of their occurrence can
be understood in the beginning of Section \ref{section3} when we formulate the four
lines problem.
\smallskip

Further, Jaming and Kellay \cite{JK} have given a unifying proof for some of the
Heisenberg uniqueness pairs corresponding to certain algebraic curves.

\begin{theorem}\cite{JK}\label{th13}
Let $\Gamma$ be any of the following curves:
\begin{enumerate}[(i)]
\item the graph of $\psi(t)=|t|^{\alpha},~ t\in\mathbb R,~ \alpha>0;$
\item a hyperbola;
\item a polygon;
\item an ellipse.
\end{enumerate}
Then there exists a set $E\subset\left(-\pi/2,\pi/2\right)\times\left(-\pi/2,\pi/2\right)$
of positive Lebesgue measure such that for $\left(\theta_{1},\theta_{2}\right)\in E,$
the pair $\left(\Gamma,L_{\theta_{1}}\cup L_{\theta_{2}}\right)$ is a HUP.
\end{theorem}

\section{ A review of the Heisenberg uniqueness pairs for the spiral,
hyperbola, circle and exponential curves}\label{section2}
In this section, we will work out some of the Heisenberg uniqueness pairs
corresponding to the spiral, hyperbola, circle and certain exponential
curves by using the basic tools of the Fourier analysis. Though, a complete
characterization for the Heisenberg uniqueness pairs corresponding to
either of the above curves is still open.
\smallskip

First, we prove that the spiral is a Heisenberg uniqueness pair for the anti-spiral.

\begin{theorem}\label{th1}
Suppose  $\Gamma=\left\{\left({e^{-t}\cos t}, {e^{-t}\sin t}\right): t\geq 0\right\}$ is a
spiral and let $\Lambda=\left\{\left({e^{s}\cos s}, {e^{s}\sin s}\right): s\leq0\right\}.$
Then $\left(\Gamma,\Lambda\right)$ is a Heisenberg uniqueness pair.
\end{theorem}
In order to prove Theorem \ref{th1}, we need the following results from \cite{BS,D}.
\begin{theorem}\label{th2}\cite{D}
Let $h$ be a bounded measurable function and $g\in L^1(\mathbb R^n).$ If $h\ast g$
vanishes identically, then $\hat h$ vanishes on the support of $\hat g.$
\end{theorem}

Let $\mathbb R^n_+=\left\{\left(x_1,\ldots, x_n\right)\in\mathbb R^n: x_j\geq0;~j=1,\ldots,n\right\}.$
The following result had appeared in the article \cite{BS} by Bagchi and Sitaram, p. 421,
as a part of the proof of Proposition $2.1.$

\begin{proposition}\label{pro1}\cite{BS}
Let $h$ be a non-zero bounded Borel measurable function which is supported on $\mathbb R^n_+.$
Then $\text{supp } \hat h=\mathbb R^n.$
\end{proposition}

\noindent{\em Proof of Theorem \ref{th1}.}
Since $\mu$ is absolutely continuous with respect to the arc length measure on $\Gamma,$
by Radon-Nikodym theorem there exists $f\in L^{1}[0,\infty)$ such that $d\mu=\sqrt{2}f(t)e^{-t}dt.$
Let $g(t)=\sqrt{2}f(t)e^{-t}.$ Then by the finiteness of $\mu,$ it follows that
$g\in L^{1}[0,\infty).$ By hypothesis, $\hat\mu\vert_\Lambda= 0$ implies
\begin{equation}\label{exp4}
\hat\mu(\xi,\eta)=\int_{0}^{\infty}e^{-i\pi e^{-t}(\xi \cos t + \eta\sin t)}d\mu(t)
=\int_{0}^{\infty}e^{-i\pi e^{(s-t)}\cos(t-s)}g(t)dt=0
\end{equation}
for all $( \xi,\eta)\in \Lambda.$ Let $H(t)=e^{ -i\pi e^{t}\cos t}\chi_{[0,\infty)}(t)$
and $G(t)=g(t)\chi_{(0,\infty)}(t).$ Then from (\ref{exp4}), we get $\hat\mu{( \xi,\eta)}=(H\ast G)(s)=0~\forall~s\in\mathbb R.$
In view of Theorem \ref{th2}, we infer that $\text{supp }\hat H\subset Z(\hat G),$
where  $Z(\hat G)$ denotes the zero set of $\hat G.$ As $H$ is a non-zero bounded
Borel measurable function supported in $[0,\infty),$ by Proposition \ref{pro1}
it follows that $\text{supp }\hat{H}=\mathbb R$ and hence $\hat G=0.$ Thus, $\mu=0.$
\smallskip

Next, we work out some of the Heisenberg uniqueness pairs corresponding to certain
exponential curves in the plane. Though, the result is true for a large class of
exponential curves, for the sake of simplicity we prove only for a particular one.

\begin{theorem}\label{th8}
Let $\alpha :\mathbb R\rightarrow\mathbb R_{+}$ be the function defined by
$\alpha(t)=e^{t^2}$ and let $\Gamma=\left\{\left(t, \alpha(t)\right):~ t\in\mathbb R\right\}.$
\smallskip

\noindent $(i)$ If $\Lambda$ is a straight line parallel to the $X$-axis.
Then $\left(\Gamma,\Lambda\right)$ is a HUP.
\smallskip

\noindent $(ii)$ Let $\Lambda=L_{1}\cup L_{2},$ where  $L_j;~j=1,2$ are
any two straight lines parallel to the $Y$-axis. Then $\left(\Gamma,\Lambda\right)$ is a HUP.
\end{theorem}

In order to prove Theorem \ref{th8}, we need the following two important results
about the uniqueness of Fourier transform. First, we state a result which
can be found in Havin and Joricke \cite{HJ}, p. 36.
\begin{lemma}\cite{HJ}\label{lemma3}
If $\varphi\in L^1(\mathbb R)$ is supported in $[0, \infty)$ and
$\int\limits_{\mathbb R}\log|\hat\varphi|\frac{dx}{1+x^2}=-\infty,$
then $\varphi=0.$
\end{lemma}
As a consequence of Lemma \ref{lemma3} we prove the following result.
\begin{lemma}\label{lemma15}
Let $g\in L^{1}(\mathbb R)$ and $\alpha :\mathbb R\rightarrow\mathbb R_{+}$ be
defined by $\alpha(t)=e^{t^2}$. Suppose $E\subset\mathbb R$ and $|E|>0.$ Then
\begin{equation}\label{exp15}
\int_{\mathbb R}e^{-i\pi y\alpha(t)}g(t)dt=0
\end{equation}
for all $y\in E$ if and only if  $g$ is an odd function.
\end{lemma}

\begin{proof}
The left hand side of Equation (\ref{exp15}) can be expressed as
\begin{eqnarray*}
I&=&\int_{-\infty}^{0}e^{-i\pi y\alpha(t)}g(t)dt+\int_{0}^{\infty}e^{-i\pi y\alpha(t)}g(t)dt\\
&=&\int_{0}^{\infty}e^{-i\pi y\alpha(t)}(g(t)+g(-t))dt\\
&=&\int_{0}^{\infty}e^{-i\pi y\alpha(t)}F(t)dt,
\end{eqnarray*}
where $F(t)=g(t)+g(-t)$ for all $t\geq0.$ Clearly $F\in L^{1}(0,\infty)$ and hence
by the change of variables $u=\alpha(t),$ we have
\begin{equation}\label{exp16}
I=\int_1^{\infty}e^{-i\pi x u}F(\sqrt{\log u})\frac{du}{2u\sqrt{\log u}}.
\end{equation}
Let $\varphi(u)={F(\sqrt{\log u})}/{2u\sqrt{\log u}}~\chi_{(1,\infty)}(u).$ Then
$\varphi\in L^{1}(\mathbb R)$ and from (\ref{exp16}) we have $I=\hat{\varphi}(y)=0$ for all $y\in E.$
Since $\hat{\varphi}$ vanishes on the set $E$ of positive Lebesgue measure, by Lemma \ref{lemma3}
it follows that $\varphi=0.$ That is, $F=0$ and hence $g$ is an odd function.
\smallskip

Conversely, if $g$ is an odd function, then (\ref{exp15}) trivially holds.
\end{proof}

\noindent {\em Proof of Theorem \ref{th8}.}
$(i)$ Since $\mu$ is supported on $\Gamma=\{(t, e^{t^2}): t\in\mathbb R\},$
there exists $f\in L^{1}(\mathbb R)$ such that $d\mu=f(t)\sqrt{1+4t^{2}e^{2t^{2}}}~dt$.
Let $g(t)=f(t)\sqrt{1 + 4t^2e^{2t^{2}}}.$ Then by the finiteness of $\mu,$ it
follows that $g\in L^{1}(\mathbb R)$ and
\[\hat\mu{(x, y)}= \int_{\mathbb R}e^{- i\pi\left(xt + ye^{t^2}\right)}g(t)dt.\]
In view of the invariance property $(i),$ we can assume that $\Lambda$ is the $X$-axis.
Hence $\hat\mu\vert_\Lambda=0$ implies that $\hat g(x)=0$ for all $x\in\mathbb R.$
Thus, we conclude that $\mu=0.$

\smallskip

\noindent $(ii)$ By invariance property $(i),$ we can assume $L_1$ is the $Y$-axis and
$L_2$ the line $x=x_o,$ where $x_o\neq0.$ Since $\hat\mu$ vanishes on $L_1,$ by Lemma
\ref{lemma15} it follows that $g$ is odd. Also, $\hat\mu$ vanishes on the line $L_2,$
implies that \[\int_{\mathbb R}e^{-i\pi(x_ot+ye^{t^2})}g(t)dt=0\] for all $y\in\mathbb R.$
In view of Lemma \ref{lemma15}, it follows that $e^{-i\pi x_ot}g(t)$ is an odd function.
Hence $e^{-i\pi x_ot}g(t)=-e^{i\pi x_ot}g(-t).$ Since $g$ is odd, it implies that
$(e^{2i\pi x_ot}-1)g(t)=0.$ As the identity $e^{2i\pi x_ot}=1$ holds only for the
countably many $t,$ we conclude that $g=0.$ Thus, $\mu=0.$

\begin{remark}\label{rk2}
Let $\alpha :\mathbb R\rightarrow\mathbb R_{+}$ be an even smooth function having
finitely many local extrema and $\Gamma=\left\{(t, \alpha(t)): t\in\mathbb R\right\}.$
Then the conclusions of Theorem \ref{th8} would also hold.
\end{remark}

Next, we work out some of the Heisenberg uniqueness pairs corresponding to the circle.
We show that (circle, spiral) is a HUP.

\smallskip

Let $\Gamma=S^1$ denote the unit circle in $\mathbb R^2$. If for $f\in L^{1}(\Gamma),$
we write $f(\theta)$ instead of $f(e^{i\theta}),$ then $f$ is a $2\pi$ periodic function
and $f\in L^{1}[0, 2\pi)$. Let $\mu$ be a finite complex-valued Borel measure in $\mathbb R^2$
which is supported on $\Gamma$ and absolutely continuous with respect to the arc length measure
on $\Gamma$. Then there exists $f\in\ L^1(S^1)$ such that  $d\mu= f(\theta)d\theta$.
Now, we prove the following result.

\begin{theorem}\label{th4}
Let $\Gamma=S^1$ and $\Lambda=\left\{\left({e^{t}\cos t},{e^{t}\sin t}\right):t\leq 0\right\}$
be the spiral. Then $\left(\Gamma,\Lambda\right)$ is a Heisenberg uniqueness pair.
\end{theorem}

\begin{proof}
Since $\mu$ is supported on the unit circle $\Gamma,$ we can write the Fourier transform of $\mu$ by
\[\hat\mu{(x,y)}= \int_{-\pi}^{\pi}e^{-i\pi (x\cos\theta + y\sin\theta)}f(\theta)d\theta.\]
Hence $\hat\mu$ can be extended holomorphically to $\mathbb C^2.$ Thus, the function $F$
defined by
\[F(z_{1}, z_{2})= \int_{-\pi}^{\pi}e^{-i\pi (z_{1}\cos\theta + z_{2}\sin\theta)}f(\theta)d\theta,\]
is  holomorphic on $\mathbb C^2.$ In particular, $ \hat\mu=F\vert_{\mathbb R^2}$
is a real analytic function. Since $\hat\mu$ vanishes on the spiral $\Lambda,$
for any line $L$ which passes through the origin, $\hat\mu\vert_{\Lambda\cap L}=0.$
As $(0,0)$ is a limit point of the set $\Lambda\cap L,$ it follows that $\hat\mu\vert_L=0.$
Since $L$ is arbitrary, we infer that $\hat\mu{(x,y)}=0$ for all $(x,y)\in\mathbb R^2.$

\smallskip

Let $S_r=\{(r\cos t, r\sin t):~ 0\leq t<2\pi\},$ where $J_{k}(r)\neq 0$ for all $k\in\mathbb Z.$
Then $\hat\mu(r\cos t,r\sin t)=0$ implies $h*f(t)=0,$ where $h(t)=e^{-i\pi r\cos t}.$
As we know that the Fourier coefficients of $h$ satisfying $\hat h(k)=i^{k}(-1)^{k}J_{k}(r),$
it follows that $\hat{f}(k)J_{k}(r)=0 $ for all $k\in\mathbb Z$. Since $J_{k}(r)\neq 0$
for all $k\in\mathbb Z,$ $\hat{f}(k)=0$ for all $k\in\mathbb Z$ and hence $f=0.$
\end{proof}

\begin{remark}\label{rk4}
A set which is determining set for any real analytic function is called $NA$ - set.
For instance, the spiral is an $NA$ - set in the plane (see \cite{PS_0}). If $\mu$
is a finite Borel measure supported on a closed and bounded curve $\Gamma,$ then
$\hat\mu$ is real analytic. Thus, $(\Gamma,\text{NA - set})$ is a Heisenberg
uniqueness pair. However, the converse is not true.

\smallskip

Hence, in view of Remarks \ref{rk5} and \ref{rk4} we expect that the exceptional sets
for the Heisenberg uniqueness pairs corresponding to the unit circle $\Gamma=S^1$
are eventually contained in the zero sets of all homogeneous harmonic polynomials
union with the countably many circles whose radii are lying in the zero set of
the certain class of Bessel functions. On the basis of these credible observations,
we are trying to find out a complete characterization of the Heisenberg uniqueness
pairs corresponding to circle which may be presented somewhere else.
\end{remark}

Next, we work out some of the Heisenberg uniqueness pairs corresponding to the
hyperbola. Though in this case, Hedenmalm and Montes-Rodríguez \cite{HR} have
found that some discrete set $\Lambda_{\alpha,\beta}$ is enough for $\left(\Gamma, \Lambda_{\alpha,\beta}\right)$
to be a Heisenberg uniqueness pair. However, our approach is to consider those
sets $\Lambda$ which are essentially a union of continuous curves and located
somewhere else than the set $\Lambda_{\alpha,\beta}.$

\begin{theorem}\label{th5}
Let $\Gamma=\left\{\left(\cosh t, \sinh t\right):t\geq0\right\}$ be a branch of
the hyperbola and $\Lambda=\left\{\left(\cosh s, -\sinh s\right): s\in\mathbb R\right\}.$
Then $\left(\Gamma, \Lambda\right)$ is a HUP.
\end{theorem}

\begin{proof}
Since $\mu$ is supported on $\Gamma,$ there exists $f\in L^{1}[0,\infty)$ such that
$d\mu=f(t)\sqrt{\cosh 2t}~dt$. If we write $g(t)=f(t)\sqrt{\cosh 2t},$ then $g\in L^{1}[0,\infty)$ and
\[\hat\mu{(x, y)}= \int_0^\infty e^{- i\pi(x\cosh t + y\sinh t)}g(t)dt.\]
By hypothesis, $\hat\mu\vert_\Lambda= 0$ implies
\begin{equation}\label{exp5}
\hat\mu{(x,y)}=\int_{0}^{\infty}e^{-i\pi \cosh(t-s)}g(t)dt=0
\end{equation}
for all $(x,y)\in \Lambda.$ Let $H(t)=e^{-i\pi {\cosh t}}\chi_{[0,\infty)}(t)$ and
$G(t)=g(t)\chi_{(0,\infty)}(t).$ Then from (\ref{exp5}) we get $\hat\mu{(x,y)}=(H*G)(s)=0$
for all $s\in\mathbb R.$ In view of Theorem \ref{th2}, it follows that $\text {supp }\hat H\subset Z(\hat G).$
Hence by Proposition \ref{pro1}, $\text{supp }\hat{H}=\mathbb R.$ Thus, we conclude that $G=0.$
\end{proof}

\begin{theorem}\label{th6}
Let $\Gamma=\left\{\left(\cosh t, \sinh t\right) : t\in\mathbb R\right\}$ and
$\Lambda=L_{1}\cup L_{2},$ where $L_j;~j=1,2$  are any two lines parallel to
the $X$-axis. Then $\left(\Gamma,\Lambda\right)$ is a HUP.
\end{theorem}
We need the following result in order to prove Theorem \ref{th6}.
\begin{lemma}\label{lemma1}
Let $g\in L^{1}(\mathbb R)$ and $E\subset\mathbb R$ such that $|E|>0.$ Then
\begin{equation}\label{exp6}
\int_{\mathbb R}e^{-i\pi x\cosh t}g(t)dt=0
\end{equation}
for all $x\in E$  if and only if $g$ is an odd function.
\end{lemma}

\begin{proof}
The left-hand side of Equation (\ref{exp6}) can be expressed as
\begin{eqnarray*}
I&=&\int_{-\infty}^{0}e^{-i\pi x\cosh t}g(t)dt+\int_{0}^{\infty}e^{-i\pi x\cosh t}g(t)dt\\
&=&\int_{0}^{\infty}e^{-i\pi x\cosh t}(g(t)+g(-t))dt  \\
&=&\int_{0}^{\infty}e^{-i\pi x\cosh t}F(t)dt,
\end{eqnarray*}
where $F(t)=g(t)+g(-t)$ for all $t\geq0.$ Clearly $F\in L^{1}(0,\infty).$
By change of variables $u=\cosh t,$ we get
\begin{equation}\label{exp7}
I=\int_1^{\infty}e^{-i\pi xu}F(\cosh ^{-1}u)\frac{du}{\sqrt{u^{2}-1}}.
\end{equation}
If we substitute $\varphi(u)=F(\cosh ^{-1}u)/\sqrt{u^{2}-1}~\chi_{(1,\infty)},$
then $\varphi\in L^{1}(\mathbb R)$ and $I=\hat{\varphi}(x)=0$ for all $x\in E.$
Hence by Lemma \ref{lemma3}, it follows that $\varphi=0.$ Thus, we infer that
$g$ is an odd function.
\smallskip

Conversely, suppose $g$ is an odd function, then (\ref{exp6}) trivially holds.
\end{proof}

\noindent {\em Proof of Theorem \ref{th6}.}
By invariance property $(i),$ we can assume that $L_{1}$ is the $X$-axis and $L_{2}$
the line $y=y_o,$ where $y_o\neq0.$ Since $\mu$ is supported on the hyperbola
$\Gamma,$ there exists $f\in L^{1}(\mathbb R)$ such that $d\mu=f(t)\sqrt{\cosh 2t}~dt$.
Let $g(t)= f(t)\sqrt{\cosh 2t},$ then $g\in L^{1}(\mathbb R).$ Hence in view of Lemma
\ref{lemma1}, $\hat\mu$ vanishes on $L_1$ implies that $g$ is an odd function. Further,
$\hat\mu\vert_{L_2}=0$ implies that
\[\int_{\mathbb R}e^{-i\pi (x\cosh t+y_o\sinh t)}g(t)dt=0\]
for all  $x\in\mathbb R.$ Then by Lemma \ref{lemma1} the function $e^{-i\pi y_o\sinh t}g(t)$
will be an odd function. Hence $e^{-i\pi y_o\sinh t}g(t)=-e^{i\pi y_o\sinh t}g(-t).$ As $g$ is
an odd function, it follows that $\left(e^{2i\pi y_o\sinh t}-1\right)g(t)=0.$ Using the fact that
$e^{2i\pi y_o\sinh t}=1$ holds only for the countably many values of $t,$ we conclude that $g=0.$

\begin{theorem}\label{th7}
Let $\Gamma=\left\{\left(\cosh t, \sinh t\right):~t\in\mathbb R\right\}$ and $\Lambda=L_{1}\cup L_{2},$
where $L_j;~j=1,2$ are any two straight lines which intersect at an angle $\alpha\in (0,\frac{\pi}{4}).$
Then $\left(\Gamma,\Lambda\right)$ is a HUP.
\end{theorem}

\begin{proof}
Without loss of generality, we can assume that $L_1$ is the $X$-axis and
$L_2=\{(s\cosh t_o, -s\sinh t_o): s\in\mathbb R\},$ where $\tan\alpha=-\tanh t_o.$
Since $\mu$ is supported on the hyperbola $\Gamma,$ as similar to Theorem \ref{th6}
there exists $g\in L^{1}(\mathbb R)$ such that $d\mu=g(t)dt.$ Suppose $\hat\mu=0$ on
$\Lambda,$ then we have \[\int_{\mathbb R}e^{-i\pi (x\cosh t+y\sinh t)}g(t)dt=0\]
for all $(x,y)\in L_2.$ This in turn implies \[\int_{\mathbb R}e^{-i\pi s\cosh t}g(t+t_o)dt=0\]
for all $s\in\mathbb R.$ In view of Lemma \ref{lemma1}, it follows that $g(t_o+\cdot)$ must be
an odd function. Since $\hat\mu$ is also vanishing on the $X$-axis, $g$ will be odd. Hence
$g\left(2t_o\pm t\right)=g(t)$ for all $t\in\mathbb R.$ That is, $g$ is a periodic function
contained in $L^1(\mathbb R).$ Thus, we conclude that $g=0.$
\end{proof}

\begin{remark}\label{rk1}
${\bf(a).}$ Let $\Gamma$ be the hyperbola and $\Lambda$ a straight line parallel to the
$X$-axis. Then $\left(\Gamma,\Lambda\right)$ is not a HUP. Consider $g=\sqrt{\cosh 2t}\sin t~\chi_{(-\pi, \pi)}$
and $d\mu=g(t)dt.$ Then $\hat\mu$ vanishes on $\Lambda.$

\smallskip

\noindent ${\bf(b).}$ We would like to mention that Theorem \ref{th7} is contained well
in the case $(ii)$ of Theorem \ref{th13} due to Jaming and Kellay \cite{JK}. However,
our approach for proof of Theorem \ref{th7} is quite different.
\end{remark}

\section{Heisenberg uniqueness pairs corresponding to the four parallel lines}\label{section3}
A  characterization of the Heisenberg uniqueness pairs corresponding to any two
parallel straight lines have been done by Hedenmalm et al. \cite{HR}. Further,
D. B. Babot \cite{Ba} has worked out an analogous result for a certain system
of three parallel lines. In this section, we prove a characterization of the
Heisenberg uniqueness pairs corresponding to a certain system of four parallel
lines. In the above case, we observe the phenomenon of interlacing of three
totally disconnected sets.
\smallskip

Let $\Gamma_o$ denote a system of four parallel lines that can be expressed as
$\Gamma_o=\mathbb R\times\{\alpha,\beta,\gamma,\delta\},$ where $\alpha<\beta<\gamma<\delta,~$
$p=(\delta-\alpha)/(\beta-\alpha)\in\mathbb N\smallsetminus\{1,2\}$ and $(\gamma-\alpha)/(\beta-\alpha)=2.$
If $(\Gamma_o, \Lambda_o)$ is a HUP, then by using invariance property $(i),$
$(\Gamma_o, \Lambda_o)$ can be reduced to $\left(\Gamma_o-(0,\alpha), \Lambda_o\right).$
Since scaling can be thought as a diagonal matrix, by using invariance property $(ii),$
$\left(\Gamma_o-(0,\alpha), \Lambda_o\right)$ can be reduced to $\left(T^{-1}(\Gamma_o-(0,\alpha)), T^\ast\Lambda_o\right),$
where $T=\text{diag}\{(\beta-\alpha),(\beta-\alpha)\}.$ Let $\Lambda=T^\ast\Lambda_o$ and
$\Gamma=T^{-1}(\Gamma_o-(0,\alpha)).$ Then $\Gamma=\mathbb R\times\{0,1,2,p\},$ where
$p\in\mathbb N$ with $p\geq3.$ Thus, $(\Gamma_o, \Lambda_o)$ is a HUP if and only if
$(\Gamma, \Lambda)$ is a HUP.

\smallskip

Before we state our main result of this section, we need to set up some necessary
notations and the subsequent auxiliary results.
\smallskip

Let $\mu$ be a finite Borel measure which is supported on $\Gamma$ and absolutely
continuous with respect to the arc length measure on $\Gamma.$ Then there exist
functions $f_k\in L^1(\mathbb R);~k=0,1,2,3$ such that
\begin{equation}\label{exp1}
d\mu=f_0(x)dxd\delta_0(y)+f_1(x)dxd\delta_1(y)+f_2(x)dxd\delta_2(y)+f_3(x)dxd\delta_p(y),
\end{equation}
where $\delta_{t}$ denotes the point mass measure at $t.$ By taking the Fourier transform
of both sides of (\ref{exp1}) we get
\begin{equation}\label{exp11}
\hat\mu(\xi,\eta)= \hat{f_0}(\xi) + e^{\pi i\eta}\hat{f_1}(\xi) +
e^{2\pi i\eta}\hat{f_2}(\xi) + e^{p\pi i\eta}\hat{f_3}(\xi).
\end{equation}

Notice that for each fixed $(\xi,\eta)\in\Lambda,$ the right-hand side of Equation $(\ref{exp11})$
is a trigonometric polynomial of degree $p$ that could have preferably some missing terms.
Therefore, it is an interesting question to find out the smallest set $\Lambda$ that determines
the above trigonometric polynomial. We observe that the size of $\Lambda$ depends on the choice
of a number of lines as well as irregular separation among themselves. That is, a larger number of
lines or value of $p$ would force smaller size of $\Lambda.$ Eventually, the problem would
become immensely difficult for a large value of $p.$
\smallskip

Observe that $\hat \mu$ is a $2$-periodic function in the second variable. Hence, for any set
$\Lambda\subset\mathbb R^2,$ it is enough to consider the set
\[\pounds(\Lambda)= \left\{(\xi,\eta) : (\xi,\eta+ 2k)\in\Lambda, \text{ for some } k\in\mathbb Z\right\}\]
for the purpose of HUP. Also, it is easy to verify that $\left(\Gamma,\Lambda\right)$ is a
HUP if and only if $(\Gamma,\overline{\pounds(\Lambda)})$ is a HUP, where $\overline{\pounds(\Lambda)}$
denotes the closure of $\pounds(\Lambda)$ in $\mathbb R^2.$ In view of the above facts, it is enough
to work with the closed set $\Lambda\subset\mathbb R^2$ which is $2$-periodic with respect
to the second variable.
\smallskip

Now, it is evident from the Riemann-Lebesgue lemma that the exponential functions,
which appeared in (\ref{exp11}), cannot be expressed as the Fourier transform
of functions in $L^1(\mathbb R).$ However, they can locally agree with the Fourier
transform of functions in $L^1(\mathbb R).$ Hence, in view of the condition
$\hat\mu\vert_\Lambda=0,$ we can classify these related exponential functions.

\smallskip

Given a set $E\subset\mathbb R$ and a point $\xi\in E,$ let $I_\xi$
denote an interval containing $\xi.$ We define three functions spaces
in the following way.

\smallskip

\noindent $\textbf{(A).}$ $ L^{E,\xi}_{loc}=\{\psi :E\rightarrow\mathbb C$
such that  $\psi(\xi)\neq0$ and there is an interval $I_\xi$ and a function
$\varphi\in L^1(\mathbb R)$ which satisfies $\psi= \hat{\varphi}$ on $I_{\xi}\cap E\}.$
\smallskip

\noindent $\textbf{(B).}$
$P^{1, 2}[ L^{E,\xi}_{loc}]=\{\psi :E\rightarrow\mathbb C$ such that there is
an interval $I_{\xi}$ and $\varphi_j\in L^1(\mathbb R);~j=0,1$ which satisfies
$\psi^2+\hat{\varphi}_1\psi+\hat{\varphi}_0=0$ on $I_{\xi}\cap E\}.$
\smallskip

Now, for $p\in\mathbb N$ with $p\geq3,$ we define the third functions space as
follows.
\smallskip

\noindent $\textbf{(C).}$  $P^{1, p}[ L^{E,\xi}_{loc}]=\{\psi :E\rightarrow\mathbb C$
such that there is an interval $I_{\xi}$ and functions $\varphi_j\in L^1(\mathbb R);~j=0,1,2$
which satisfy $\psi^p+\hat{\varphi}_2\psi^2+\hat{\varphi}_1\psi+\hat{\varphi}_0=0$
on $I_{\xi}\cap E\}.$
\smallskip

We will frequently use the following Wiener's lemma that plays a key role in the
rest part of the arguments for proofs.

\begin{lemma}\label{lemma4}\cite{K}
Let $\psi\in L^{E,\xi}_{loc}$ and $\psi(\xi)\neq0.$ Then $1/{\psi}\in  L^{E,\xi}_{loc}.$
\end{lemma}
For more details, see \cite{K}, p.57.
\smallskip

In view of Lemma \ref{lemma4}, we derive the following relation among the sets which are
described by $\textbf{(A)},~\textbf{(B)}$ and $\textbf{(C)}.$ We would like to mention that
the integral choice of $p$ in Lemma \ref{lemma5} has been considered for a convenience.
\begin{lemma}\label{lemma5}
 For $p\geq3,$ the following inclusions hold.
\begin{equation}\label{exp3}
L^{E,\xi}_{loc}\subset P^{1,2}[ L^{E,\xi}_{loc}]\subset P^{1, p}[ L^{E,\xi}_{loc}].
\end{equation}
\end{lemma}

\begin{proof}
$\bf{(a)} $ If $\psi\in L^{E,\xi}_{loc},$ then by the Wiener's lemma $1/{\psi}\in  L^{E,\xi}_{loc}.$
By definition, there exist intervals $I_1, I_2$ containing $\xi$ and functions $f,g\in L^1(\mathbb R)$
such that $\psi=\hat{f}$ on $I_1\cap E$ and $\frac{1}{\psi}=\hat{g}$  on $I_2\cap E.$
Hence we can extract an interval $I_3\subset I_1\cap I_2$ containing $\xi$ such that
$\psi^2=\frac{\hat f}{\hat g}$ on $I_3\cap E.$ As $\hat g(\xi)\neq 0,$ there exists
an interval $I_4$ containing $\xi$ and a function $h\in L^1(\mathbb R)$ such that
$\frac{1}{\hat g}= \hat h$ on $ I_4\cap E.$ Further, we can extract an interval
$I_5\subset I_3\cap I_4$ containing $\xi$ such that
\begin{equation}\label{exp12}
\psi^2=\hat f~\hat h=\widehat{f\ast h}=\hat\varphi
\end{equation}
on $I_5\cap E,$ where $\varphi=f\ast h\in L^1(\mathbb R).$ This implies $\psi^2\in L^{E,\xi}_{loc}.$
Hence by the induction argument, it can be shown that $\psi^p\in L^{E,\xi}_{loc},$ whenever $p\in\mathbb N.$
Now, consider a function $f_0\in L^1(\mathbb R)$ such that $I_1\subset\text{supp }\hat{f}_0.$
Since $\psi=\hat{f}$ on $I_1\cap E,$ it follows that
\begin{equation}\label{exp13}
\hat{f}_0\psi= \hat{f}_0\hat{f}=\widehat{f_0\ast f}
\end{equation}
on $I_1\cap E.$ Hence from (\ref{exp12}) and (\ref{exp13}) we conclude that
\begin{equation}\label{exp14}
 \psi^2+\hat{\varphi}_1\psi+\hat{\varphi}_0=0
\end{equation}
on $I_\xi\cap E,$ where $I_\xi\subset I_1\cap I_5,~\varphi_0=-(f_0\ast f+\varphi)$ and
$\varphi_1=f_0.$ Thus, $\psi\in P^{1, 2}[ L^{E,\xi}_{loc}].$ By applying induction, we can
show that $\psi^p+\hat{\varphi}_1\psi+\hat{\varphi}_0=0,$ whenever $p\in\mathbb N.$

\smallskip

\noindent $\bf{(b)} $ If $\psi\in P^{1, 2}[ L^{E,\xi}_{loc}],$ then there exists an interval
$I_\xi$ containing $\xi$ and functions $f,g\in L^1(\mathbb R)$ such that
\begin{equation}\label{exp17}
\psi^2+\hat{f} \psi+\hat g=0
\end{equation}
on $I_\xi\cap E.$
Now, consider a function $f_0\in L^1(\mathbb R)$ such that $I_\xi\subset\text{supp }\hat{f}_0.$
After multiplying (\ref{exp17}) by $\psi$ and $\hat{f}_0$ separately
and adding the resultant equations, we can write
\[\psi^3+\left(\hat{f}_0+ \hat f\right)\psi^2+\left(\hat{f}_0\hat f+
\hat g\right)\psi+\hat{f}_0\hat g=0.\]
Hence for the appropriate choice of $\varphi_j;~j=0,1,2,$ we have
\begin{equation}\label{exp141}
\psi^3+\hat{\varphi}_2\psi^2+\hat{\varphi}_1\psi+\hat{\varphi}_0=0
\end{equation}
on $I_\xi\cap E.$ Further by induction, it follows that
$\psi^p+\hat{\varphi}_2\psi^2+\hat{\varphi}_1\psi+\hat{\varphi}_0=0$
on $I_\xi\cap E,$ whenever $p\in\mathbb N.$ Thus, $\psi\in P^{1, p}[ L^{E,\xi}_{loc}].$
\end{proof}

Let $\Pi(\Lambda)$ be the projection of $\Lambda$ on $\mathbb R\times\{0\}.$ For
$\xi\in\Pi(\Lambda),$ we denote the corresponding image on the $\eta $ - axis by
\[\varSigma_{\xi}= \{\eta\in [0,2) : (\xi,\eta)\in\Lambda\}.\]

Now, we require analyzing the set $\Pi(\Lambda)$ to know its basic geometrical
structure in accordance with the Heisenberg uniqueness pair. Since it is expected
that the set $\varSigma_{\xi}$ may consist one or more image points depending upon
the order of its winding, the set $\Pi(\Lambda)$ can be decomposed into the
following four disjoint sets. For the sake of convenience, we denote $F_o=\{0,1,2,3\}.$

\smallskip

\noindent ${\bf{(P_1).}}$ $\Pi^1(\Lambda)=\{\xi\in\Pi(\Lambda):\text{ there is a unique }\eta_0\in\varSigma_{\xi}\}.$
\smallskip

\noindent ${\bf{(P_2).}}$ $\Pi^2(\Lambda)=\{\xi\in\Pi(\Lambda):\text{ there are only two distinct }\eta_j\in \varSigma_{\xi};~j=0,1\}.$
\smallskip

In order to describe the rest of the two partitioning sets, we will use the
notion of symmetric polynomial. For each $k\in\mathbb Z_+,$ the complete homogeneous
symmetric polynomial $H_k$ of degree $k$ is the sum of all monomials of degree
$k.$ That is,
\[H_k\left(x_1,\ldots,x_n\right)=\sum\limits_{l_1+\cdots+l_n=k;\;l_i\geq 0}x_1^{l_1}\ldots x_n^{l_n}.\]
For more details, we refer to \cite{MIG}.
\smallskip

Consider four distinct image points $\eta_j\in[0,2)$ and denote $a_j=e^{\pi i\eta_j};~j\in F_0.$
For $p\geq3,$ we define the remaining two sets as follows:
\smallskip

\noindent ${\bf{(P_3).}}$ $\Pi^3(\Lambda)=\{\xi\in\Pi(\Lambda):$ there are at least
three distinct $\eta_j\in\varSigma_{\xi}$ for $j=0,1,2$ and if there is another
$\eta_3\in \varSigma_{\xi},$ then $H_{p-2}(a_0,a_{1},a_2)=H_{p-2}(a_0,a_{1},a_{3})\}.$
\smallskip

\noindent ${\bf{(P_4).}}$ $\Pi^4(\Lambda)=\{\xi\in\Pi(\Lambda):$ there are at least four distinct
 $\eta_j\in\varSigma_{\xi};~j\in F_o$ which satisfy
$H_{p-2}(a_0,a_{1},a_2)\neq H_{p-2}(a_0,a_{1},a_{3})\}.$

\smallskip

In this way, we get the desired decomposition as $\Pi(\Lambda)=\bigcup\limits_{j=1}^4\Pi^j(\Lambda).$

\smallskip

Now, for three distinct image points $\eta_j\in [0,2);~j=0,1,2,$ denote
$a=e^{\pi i\eta_0},$ $b=e^{\pi i\eta_1}$ and $c=e^{\pi i\eta_2}.$ Consider
the system of equations $A_\xi^3X=B_\xi^p,$ where
\begin{equation}\label{exp21}
A_\xi^3= \left( \begin{array}{ccc}
1 & a & a^2 \\
1 & b & b^2 \\
1 & c & c^2 \end{array} \right),
\end{equation}
$X_\xi=(\tau_0,\tau_1,\tau_2)$ and $B_\xi^p=-\left(a^p,b^p,c^p\right).$ Since
$\det A_\xi^3=(a-b)(b-c)(c-a)\neq0,$ $A_\xi^3X=B_\xi^p$ has a unique solution.
A simple calculation gives

\begin{eqnarray}\label{exp19}
 \tau_o&=&-abcH_{p-3}(a,b,c),\\
\tau_1&=&H_{p-1}(a,b,c)-\left(a^{p-1}+b^{p-1}+c^{p-1}\right)+\sum\limits_{l+m+n=p-1\atop {l,m,n\geq1}}a^lb^mc^n,\nonumber\\
\tau_2&=&-H_{p-2}(a,b,c).\nonumber
\end{eqnarray}

Since measure in the question is supported on a certain system of four parallel lines
and the exponential functions which have appeared in (\ref{exp11}) can locally agree
with the Fourier transform of some functions in $L^1(\mathbb R),$ the following sets
sitting in $\Pi(\Lambda)$ seems to be dispensable in the process of getting the
Heisenberg uniqueness pairs.
\smallskip

\noindent $\bf{(P_{1^\ast})}.$
As each $\xi\in\Pi^1(\Lambda)$ has a unique image in $\varSigma_{\xi},$ we can define
a function $\chi_0$ on $\Pi^1(\Lambda)$ by $\chi_0(\xi)= e^{\pi i\eta_0},$ where
$\eta_0=\eta_0(\xi)\in \varSigma_{\xi}.$ Now, the first dispensable set can be defined by
\[\Pi^{1^\ast}(\Lambda)=\left\{\xi\in\Pi^1(\Lambda):\chi_0\in P^{1, p}[L^{\Pi^1(\Lambda),\xi}_{loc}]\right\}.\]

Next, for $\xi\in\Pi^2(\Lambda),$ let $\chi_j(\xi)= e^{\pi i\eta_j},$ where $\eta_j=\eta_j(\xi)\in \varSigma_{\xi};~j=0,1.$

\smallskip

\noindent $\bf{(P_{2^\ast}').}$
Since each $\xi\in\Pi^2(\Lambda)$ has two distinct image points in $\varSigma_{\xi},$
we define two functions $\delta_j$ on $\Pi^2(\Lambda);j=0,1$ such that  $X_\xi=\left(\delta_0(\xi),\delta_1(\xi)\right)$
is the solution of $A^2_\xi X_\xi=B^2_\xi,$ where
{\[A^2_\xi= \left( \begin{array}{ccc}
1 & {\chi_0(\xi)} \\
1 & {\chi_1(\xi)}
\end{array} \right)\]}
and $B^2_\xi=-\left({\chi_0(\xi)}^2,{\chi_1(\xi)}^2\right).$ In this way, an auxiliary dispensable
set can be defined by
\[\Pi^{2^\ast}(\Lambda)=\left\{\xi\in\Pi^2(\Lambda):\delta_j\in L^{\Pi^2(\Lambda),\xi}_{loc}; ~j=0,1\right\}.\]

\noindent $\bf{(P_{2^\ast}'').}$ Further, we define three functions $\rho_j$ on $\Pi^2(\Lambda);~j=0,1,2$
such that  $X_\xi=\left(\rho_0(\xi),\rho_1(\xi),\rho_2(\xi)\right)$ becomes a solution of  $A^p_\xi X_\xi=B^p_\xi,$
where
{\[A^p_\xi= \left( \begin{array}{ccc}
1 & {\chi_0(\xi)} & {\chi_0(\xi)}^2\\
1 & {\chi_1(\xi)} & {\chi_1(\xi)}^2
\end{array} \right)\]}
and $B^p_\xi=-\left({\chi_0(\xi)}^p,{\chi_1(\xi)}^p\right).$ Hence the second dispensable set can be
defined by
\[\Pi^p_{2^\ast}(\Lambda)=\left\{\xi\in\Pi^2(\Lambda):\rho_j\in L^{\Pi^2(\Lambda),\xi}_{loc}; ~j=0,1,2\right\}.\]

For $\xi\in\Pi^3(\Lambda),$ let $\chi_j(\xi)= e^{\pi i\eta_j},$ where $\eta_j=\eta_j(\xi)\in \varSigma_{\xi};~j=0,1,2.$

\smallskip

\noindent $\bf{(P_{3^\ast}').}$ For each $\xi\in\Pi^3(\Lambda)$ has three distinct image
points in $\varSigma_{\xi},$ we define three functions $e_j$ on $\Pi^3(\Lambda);~j=0,1,2$
such that $X_\xi=\left(e_0(\xi),e_1(\xi),e_2(\xi)\right)$ is the solution of
$A^3_\xi X_\xi=B^3_\xi,$ where $A^3_\xi$ is the matrix given by Equation (\ref{exp21})
and $B^3_\xi=-\left({\chi_0(\xi)}^3,{\chi_1(\xi)}^3,{\chi_2(\xi)}^3\right).$
Hence another auxiliary dispensable set can be defined by
\[\Pi^{3^\ast}(\Lambda)=\left\{\xi\in\Pi^3(\Lambda):e_j\in L^{\Pi^3(\Lambda),\xi}_{loc}; ~j=0,1,2\right\}.\]
\smallskip

\noindent $\bf{(P_{3^\ast}'').}$ Once again we define three functions $\tau_j$ on
$\Pi^3(\Lambda);~j=0,1,2$ such that $X_\xi=\left(\tau_0(\xi),\tau_1(\xi),\tau_2(\xi)\right)$
is the solution of $A^3_\xi X_\xi=B^p_\xi,$ where
$B^p_\xi=-\left({\chi_0(\xi)}^p,{\chi_1(\xi)}^p,{\chi_2(\xi)}^p\right).$
Hence the third dispensable set can be defined by
\[\Pi^p_{3^\ast}(\Lambda)=\left\{\xi\in\Pi^3(\Lambda):\tau_j\in L^{\Pi^3(\Lambda),\xi}_{loc};~j=0,1,2\right\}.\]

Now, we prove the following two lemmas that speak about a sharp contrast in the pattern
of dispensable sets as compared to dispensable sets which appeared in two lines and three
lines results. That is, a larger value of $p$ will increase the size of dispensable sets in
case of four lines problem. Further, we observe that dispensable sets are eventually those
sets contained in $\Pi(\Lambda)$ where we could not solve Equation (\ref{exp11}).
For more details, we refer to \cite{Ba,HR}.

\begin{lemma}\label{lemma99}
For $p\geq3,$ the following inclusion holds.
\[\Pi^{2^\ast}(\Lambda)\subset\Pi^p_{2^\ast}(\Lambda).\]
\end{lemma}
\begin{proof}
If $\xi_o\in\Pi^{2^\ast}(\Lambda),$ then $\delta_j\in L^{\Pi^2(\Lambda),\xi_o}_{loc}.$
Hence there exists an interval $I_{\xi_o}$ containing $\xi_o$ and  $\varphi_j\in L^1(\mathbb{R})$
such that $\delta_j=\hat{\varphi}_j;~j=0,1$ satisfy \[\hat\varphi_0+\hat\varphi_1\chi_j+{\chi_j}^2=0\]
on $I_{\xi_o}\cap\Pi^2(\Lambda),$ whenever $j=0,1.$ Now, by the similar iteration as in the proof
of Lemma \ref{lemma5}$(b),$ we infer that there exist a common set of $\psi_j\in L^1(\mathbb{R});~j=0,1,2$
such that
\[\hat{\psi}_0+\hat{\psi}_1\chi_j+\hat{\psi}_2{\chi_j}^2+{\chi_j}^p=0\]
on $I_{\xi_o}\cap\Pi^2(\Lambda),$ whenever $j=0,1.$ If we denote $\hat{\psi}_j=\rho_j,$
then it is easy to see that $\xi_o\in\Pi^p_{2^\ast}(\Lambda).$
\end{proof}

\begin{lemma}\label{lemma100}
For $p\geq3,$ the following inclusion holds.
\[\Pi^{3^\ast}(\Lambda)\subseteq\Pi^p_{3^\ast}(\Lambda).\] Moreover, equality
holds for $p=3.$
\end{lemma}
\begin{proof}
If $\xi_o\in\Pi^{3^\ast}(\Lambda),$ then $e_j\in L^{\Pi^3(\Lambda),\xi_o}_{loc}.$
Hence there exists an interval $I_{\xi_o}$ containing $\xi_o$ and  $\varphi_j\in L^1(\mathbb{R})$
such that $e_j=\hat{\varphi}_j;~j=0,1,2$ satisfy
\[\hat{\varphi}_0+\hat{\varphi}_1\chi_j+\hat{\varphi}_2{\chi_j}^2+{\chi_j}^3=0\]
on $I_{\xi_o}\cap\Pi^3(\Lambda),$ whenever $j=0,1,2.$ By the similar iteration as in the proof
of Lemma \ref{lemma5}$(b),$ it follows that there exist $\psi_j\in L^1(\mathbb{R});~j=0,1,2$
such that \[\hat{\psi}_0+\hat{\psi}_1\chi_j+\hat{\psi}_2{\chi_j}^2+{\chi_j}^p=0\]
on $I_{\xi_o}\cap\Pi^3(\Lambda),$ whenever $j=0,1,2.$ If we denote $\hat{\psi}_j=\tau_j,$
then it is easy to verify that $\xi_o\in\Pi^p_{3^\ast}(\Lambda).$
\end{proof}

On the basis of structural properties of the dispensable sets, we observe that
these sets are essentially minimizing the size of projection $\Pi(\Lambda).$
Now, we can state our main result of this section about the Heisenberg uniqueness
pairs corresponding to the above described system of four parallel straight lines.

\begin{theorem}\label{th17}
Let $\Gamma=\mathbb R\times\{0,1,2,p\},$ where $p\in\mathbb N$ and $p\geq3.$
Let $\Lambda\subset\mathbb R^2$ be a closed set which is $2$-periodic with
respect to the second variable. Suppose $\Pi(\Lambda)$ is dense in $\mathbb R.$
If $\left(\Gamma,\Lambda\right)$ is a Heisenberg uniqueness pair, then the set
\[\widetilde{\Pi}(\Lambda)= {\Pi^4(\Lambda)}\bigcup\limits_{j=0}^2\left[{{\Pi^{(3-j)}(\Lambda)}\smallsetminus{\Pi^{{(3-j)}^\ast}(\Lambda)}}\right] \]
is dense in $\mathbb R.$ Conversely, if the set
\[\widetilde{\Pi_p}(\Lambda)= {\Pi^4(\Lambda)}\cup\left[{{\Pi^{3}(\Lambda)}\smallsetminus{\Pi^p_{{3}^\ast}(\Lambda)}}\right]\cup\left[{{\Pi^{2}(\Lambda)}\smallsetminus{\Pi^p_{{2}^\ast}(\Lambda)}}\right] \cup\left[{{\Pi^{1}(\Lambda)}\smallsetminus{\Pi^{{1}^\ast}(\Lambda)}}\right]\]
is dense in $\mathbb R,$ then $\left(\Gamma,\Lambda\right)$ is a Heisenberg uniqueness pair.
\end{theorem}

\begin{remark}\label{rk7}
In view of Lemma \ref{lemma99}, we infer that $\Pi^{2^\ast}(\Lambda)$ is a proper subset of
$\Pi^p_{2^\ast}(\Lambda)$ for any $p\geq3.$ However, for $p=3,$ Lemma \ref{lemma100} yields
$\Pi^{3^\ast}(\Lambda)=\Pi^p_{3^\ast}(\Lambda).$ Hence for any $p\geq3,$ the set $\widetilde{\Pi_p}(\Lambda)$
is properly contained in $\widetilde\Pi(\Lambda).$ Thus, an analogous result for four lines
problem as compared to three lines result is still open.
\end{remark}

We need the following two lemmas which are required to prove the necessary
part of Theorem \ref{th17}. The main idea behind these lemmas is to pull down
an interval from some of the partitioning sets of the projection $\Pi(\Lambda).$
The above argument helps to negate the assumption that $\widetilde\Pi(\Lambda)$
is not dense in $\mathbb R.$
\begin{lemma}\label{lemma6}
Suppose $I$ is an interval such that $I\cap\Pi^{2^\ast}(\Lambda)$ is dense in $I.$
Then there exists an interval $I'\subset I$ such that $I'\subset\bigcup\limits_{j=2}^4\Pi^j(\Lambda).$
\end{lemma}
\begin{proof}
If $\bar\xi\in I\cap\Pi^{2^\ast}(\Lambda),$ then $\delta_j\in L^{\Pi^2(\Lambda),\bar\xi}_{loc};~j=0,1.$
By hypothesis, $I\cap\Pi^{2^\ast}(\Lambda)$ is dense in $I,$
therefore there exists an interval $I_{\bar\xi}\subset I$ containing $\bar\xi$
such that $\delta_j$ can be extended continuously on $I_{\bar\xi}.$ In addition, $\delta_1$
satisfies
\begin{equation}\label{exp20}
|\delta_1(\bar\xi)|=\left|e^{\pi i\bar\eta_0}+e^{\pi i\bar\eta_1}\right|<2,
\end{equation}
whenever $\bar\xi\in I\cap\Pi^{2^\ast}(\Lambda).$ Since $\delta_1$ is continuous on $I_{\bar\xi},$
we can extract an interval $I'\subset I_{\bar\xi}$ containing $\bar\xi$ such that
$|\delta_1(\xi)|<2~\text{for all}~\xi\in I'.$

\smallskip

 Consequently, $I'\cap\Pi^{2^\ast}(\Lambda)$
is dense in $I'.$ Now for $\xi\in I',$ there exists a sequence $\xi_n\in I'\cap\Pi^{2^\ast}(\Lambda)$
such that $\xi_n\rightarrow\xi.$ Hence the corresponding image sequences
$\eta_j^{(n)}\in\varSigma_{\xi_n}\subseteq [0, 2)$  will have convergent
subsequences, say $\eta_j^{(n_k)}$ which converge to $\eta_j; ~j=0,1.$
Since the set $\Lambda$ is closed, $(\xi, \eta_j)\in\Lambda$ for
$j=0,1.$ Now, we only need to show that $\eta_0\neq\eta_1.$ If possible,
suppose $\eta_0=\eta_1,$ then by the continuity of $\delta_1$ on $I',$
it follows that $|\delta_1(\xi_n)|\rightarrow|\delta_1(\xi)|.$ However,
\[|\delta_1(\xi_{n_k})|=\left|e^{\pi i\eta_0^{(n_k)}}+
e^{\pi i\eta_1^{(n_k)}}\right|\rightarrow 2.\]
That is, $|\delta_1(\xi)|=2,$ which contradicts the fact that
$|\delta_1(\xi)|<2~\text{for all}~\xi\in I'.$ Thus, we infer
that $I'\subset\bigcup\limits_{j=2}^4\Pi^j(\Lambda).$
\end{proof}

\begin{lemma}\label{lemma7}
Let $I$ be an interval such that $I\cap\Pi^{3^\ast}(\Lambda)$
is dense in $I.$ Then there exists an interval $I'\subset I$ such that
$I'$ is contained in $\Pi^{3^\ast}(\Lambda)\cup\Pi^4(\Lambda).$
\end{lemma}

\begin{proof}
Let $\bar\xi\in I\cap\Pi^{3^\ast}(\Lambda),$ then $e_j\in L^{\Pi^3(\Lambda),\bar \xi}_{loc};~j=0,1,2.$
For $p=3,$ Equation (\ref{exp19}) yields \[(e_0,e_1,e_2)=\left(-abc,(ab+bc+ca),-(a+b+c)\right),\]
where $(a, b, c)=(\chi_0,\chi_1,\chi_2).$ Hence $e_j;~j=0,1,2$ are constant multiples of the
elementary symmetric polynomials. Now, we define a function $\rho$ on $\Pi^3(\Lambda)$ by
\[\rho=\left(a^3(b-c)+b^3(c-a)+c^3(a-b)\right)^2.\]
Since $\rho$ is a symmetric polynomial in $a,b,c,$ by the fundamental theorem of symmetric polynomials,
$\rho$ can be expressed as a polynomial in $e_j;~j=0,1,2.$ Moreover, $\rho(\bar\xi)\neq0.$ Hence
it follows that $\rho\in L^{\Pi^3(\Lambda),\bar \xi}_{loc}.$ By hypothesis, $I\cap\Pi^{3^\ast}(\Lambda)$
is dense in $I,$ there exists an interval $I_{\bar\xi}\subset I$ containing $\bar\xi$
such that $\rho$ can be continuously extended on $I_{\bar\xi}.$ Thus, by continuity
of $\rho$ on $I_{\bar\xi},$ there exists an interval $J\subset I_{\bar\xi}$ containing
$\bar\xi$ such that $\rho(\xi)\neq0$ for all $\xi\in J.$

\smallskip

Consequently, $J\cap\Pi^{3^\ast}(\Lambda)$ is dense in $J$ and hence for $\xi\in J,$
there exists a sequence $\xi_n\in J\cap\Pi^{3^\ast}(\Lambda)$ such that $\xi_n\rightarrow\xi.$
Thus, the corresponding image sequences $\eta_j^{(n)}\in\varSigma_{\xi_n}\subseteq [0, 2)$
will have convergent subsequences, say $\eta_j^{(n_k)}$ which converge to $\eta_j; ~j=0,1,2$.
Since the set $\Lambda$ is closed, $(\xi, \eta_j)\in\Lambda$ for $j=0,1,2.$

\smallskip

Next, we claim that all of $\eta_j;~j=0,1,2$ are distinct. On the contrary,
suppose all are equal or any two of them are equal. Then by the continuity of $\rho$
on $J,$ it follows that $\rho(\xi)=0,$ which contradicts the fact that $\rho(\xi)\neq0$
for all $\xi\in J.$ Hence we infer that $ J\subset\bigcup\limits_{j=3}^4\Pi^j(\Lambda).$
Further, using the facts that $e_j\in L^{\Pi^3(\Lambda),\bar\xi}_{loc}$ and $J\cap\Pi^{3^\ast}(\Lambda)$
is dense in $J,$ $e_j$ can be extended continuously on an interval $I'\subset J$ containing
$\bar\xi$ such that $e_j(\xi)\neq0$ for all $\xi\in I'.$ That is, if $\xi\in I'\cap\Pi^3(\Lambda),$
then $e_j\in L^{\Pi^3(\Lambda),\xi}_{loc}$ and hence $\xi\in\Pi^{3^\ast}(\Lambda).$ Thus,
we conclude that  $I'\subset\Pi^{3^\ast}(\Lambda)\cup\Pi^4(\Lambda).$
\end{proof}

\noindent {\em Proof of Theorem \ref{th17}.}
We first prove the sufficient part of Theorem \ref{th17}. Suppose the set $\widetilde{\Pi_p}(\Lambda)$
is dense in $\mathbb R.$ Then we show that $(\Gamma, \Lambda)$ is a Heisenberg uniqueness pair.
For $\hat\mu\vert_\Lambda=0,$ we claim that $\hat{f_k}\vert_{\widetilde{\Pi_p}(\Lambda)}=0,$
whenever $k\in F_o.$ Since $\hat{f_k}$ is a continuous function which vanishes on a dense set
$\widetilde{\Pi_p}(\Lambda),$ it follows that $\hat{f_k}\equiv0$ for all $k\in F_o.$ Thus, $\mu=0.$

\smallskip

As the projection $\Pi(\Lambda)$ is decomposed into the four pieces, the proof of the
above assertion will be carried out in the following four cases.

\smallskip

\noindent $\bf{(S_1).} $ If $\xi\in\Pi^4(\Lambda),$ then there exist at least four distinct
$\eta_j\in\varSigma_\xi$ such that $\hat\mu(\xi,\eta_j)=0$ for all $j\in F_o.$ Hence $\hat{f_k}(\xi); ~k\in F_o$ satisfy
a homogeneous system of four equations. As $\xi\in\Pi^4(\Lambda),$ by using the
property that $H_{p-2}(a_0,a_{1},a_2)\neq H_{p-2}(a_0,a_{1},a_{3}),$ we infer
that $\hat{f_k}(\xi)=0$ for all $k\in F_o.$

\smallskip

\noindent $\bf{(S_2).} $ If $\xi\in\Pi^3(\Lambda),$ then there exist at least three
distinct $\eta_j\in\varSigma_\xi$ which satisfy $\hat\mu(\xi,\eta_j)=0;$ $j=0,1,2.$
If $\hat{f_3}(\xi)=0,$ then we get $\hat{f_k}(\xi)=0$ for $k=0,1,2.$
On the other hand if $\hat{f_3}(\xi)\neq0,$ then we can substitute
\begin{equation}\label{exp2}
\hat{f_j}(\xi)=\tau_j(\xi)\hat{f_3}(\xi),
\end{equation}
where $\tau_j$ are defined on $\Pi^3(\Lambda)$ for $j=0,1,2.$ Hence $X_{\xi}=\left(\tau_0(\xi),\tau_1(\xi),\tau_2(\xi)\right)$
will satisfy the system of equations $A^3_\xi X_\xi=B^p_\xi.$ By applying the Wiener lemma to Equations (\ref{exp2}),
we infer that $\tau_j\in L^{\Pi^3(\Lambda),\xi}_{loc};~j=0,1,2.$ That is, $\xi\in\Pi^p_{3^\ast}(\Lambda).$
Thus for $\xi\in{\Pi^{3}(\Lambda)\smallsetminus\Pi^p_{3^\ast}(\Lambda)},$ we conclude that $\hat{f_k}(\xi)=0$
for all $k\in F_o.$

\smallskip

\noindent $\bf{(S_3).} $ If $\xi\in\Pi^2(\Lambda),$ then there exist two distinct
$\eta_j\in\varSigma_\xi$ for which $\hat\mu(\xi,\eta_j)=0,$ whenever $j=0, 1.$
That is,
\begin{equation}\label{exp9}
\hat{f_0}(\xi)+\chi_j(\xi)\hat{f_1}(\xi)+\chi_j^2(\xi)\hat{f_2}(\xi)+\chi_j^p(\xi)\hat{f_3}(\xi)=0,
\end{equation}
where $\chi_j(\xi)=e^{\pi i\eta_j};~j=0, 1.$ If $\hat{f_3}(\xi)\neq0,$
then by applying the Wiener lemma to Equations (\ref{exp9}), it follows that
$\xi\in\Pi^p_{2^\ast}(\Lambda).$ That is, if $\xi\in{\Pi^2(\Lambda)\smallsetminus\Pi^p_{2^\ast}(\Lambda)},$
then $\hat{f_3}(\xi)=0.$

\smallskip

Further, if $\hat{f_3}(\xi)=0$ and $\hat{f_2}(\xi)\neq0,$ then an application of the
Wiener lemma to Equations (\ref{exp9}), it follows that $\xi\in\Pi^{2^\ast}(\Lambda).$
By Lemma \ref{lemma99}, $\xi\in\Pi^p_{2^\ast}(\Lambda).$ Thus for
$\xi\in{\Pi^2(\Lambda)\smallsetminus\Pi^p_{2^\ast}(\Lambda)},$
we infer that $\hat{f_k}(\xi)=0$ for all $k\in F_o.$

\smallskip

\noindent $\bf{(S_4).} $ If $\xi\in\Pi^1(\Lambda),$ then there exists a unique
$\eta_0\in\varSigma_\xi$ for which $\hat\mu(\xi,\eta_0)=0.$ That is,
\begin{equation}\label{exp10}
\hat{f_0}(\xi)+\chi_0(\xi)\hat{f_1}(\xi)+\chi_{0}^2(\xi)\hat{f_2}(\xi)+\chi_0^p(\xi)\hat{f_3}(\xi)=0,
\end{equation}
where $\chi_0(\xi)=e^{\pi i\eta_0}.$ If $\hat{f_3}(\xi)\neq0,$ then by applying the Wiener lemma to
Equation (\ref{exp10}), it implies that $\chi_0\in P^{1,p}[L^{\Pi^1(\Lambda),\xi}_{loc}].$ That is,
$\xi\in\Pi^{1^\ast}(\Lambda).$ Thus for $\xi\in{\Pi^1(\Lambda)\smallsetminus\Pi^{1^\ast}(\Lambda)},$
we have $\hat{f_3}(\xi)=0.$

\smallskip

Further, if $\hat{f_3}(\xi)=0$ and $\hat{f_2}(\xi)\neq0,$ then an application of the Wiener lemma to
Equation (\ref{exp10}), yields $\chi_0\in P^{1,2}[L^{\Pi^1(\Lambda),\xi}_{loc}].$ By Lemma
\ref{lemma5}, it follows that  $\xi\in\Pi^{1^\ast}(\Lambda).$ That is, if
$\xi\in{\Pi^1(\Lambda)\smallsetminus\Pi^{1^\ast}(\Lambda)},$ then $\hat{f_k}(\xi)=0$
for $k=2,3.$

\smallskip

Finally, if $\hat{f_k}(\xi)=0$ for $k=2,3$ and $\hat{f_1}(\xi)\neq0,$ then by applying the Wiener lemma
to Equation (\ref{exp10}), we infer that $\chi_0\in L^{\Pi^1(\Lambda),\xi}_{loc}.$ By Lemma \ref{lemma5},
it follows that $\xi\in\Pi^{1^\ast}(\Lambda).$ Thus for $\xi\in{\Pi^1(\Lambda)\smallsetminus\Pi^{1^\ast}(\Lambda)},$
we conclude that $\hat{f_k}(\xi)=0$ for all $k\in F_o.$

\smallskip

Now, we prove the necessary part of Theorem \ref{th17}. Suppose $\left(\Gamma, \Lambda\right)$
is a Heisenberg uniqueness pair. Then we claim that the set $\widetilde\Pi(\Lambda)$ is dense in
$\mathbb R.$  We observe that this is possible if the dispensable sets $\Pi^{j^\ast}(\Lambda);~j=1,2,3$
interlace to each other, though these sets are disjoint among themselves.

\smallskip

If possible, suppose $\widetilde\Pi(\Lambda)$ is not dense in $\mathbb R.$ Then there
exists an open interval $I_o\subset\mathbb R$ such that $I_o\cap\widetilde\Pi(\Lambda)$
is empty. This in turn implies that
\begin{equation}\label{exp8}
\Pi(\Lambda)\cap I_o=\left(\bigcup_{j=1}^3\Pi^{j^\ast}(\Lambda)\right)\cap I_o.
\end{equation}
Thus from (\ref{exp8}), it follows that $I_o$ intersects only the dispensable sets.
Now, the remaining part of the proof of Theorem \ref{th17} is a consequence of the
following two lemmas which provide the interlacing property of the dispensable sets
$\Pi^{j^\ast}(\Lambda);~j=1,2,3.$

\begin{lemma}\label{lemma8}
There does not exist any interval $J\subset I_o$ such that
$\Pi(\Lambda)\cap J$ is contained in $\Pi^{j^\ast}(\Lambda);~j=1,2,3.$
\end{lemma}
\begin{proof}
On the contrary, suppose there exists an interval $J\subset I_o$ such that
$\Pi(\Lambda)\cap J\subset\Pi^{j^\ast}(\Lambda),~\text{for some}~j\in\{1,2,3\}.$
Since $\Pi^{j^\ast}(\Lambda);~j=1,2,3$ are disjoint among themselves, there could be
three possibilities.
 \smallskip

\noindent $\bf{(a).}$ If $\xi\in\Pi(\Lambda)\cap J\subset\Pi^{1^\ast}(\Lambda),$
then $\chi_0\in P^{1,p}[L^{\Pi^1(\Lambda),\xi}_{loc}].$ Hence there exists
an interval $I_\xi\subset J$ containing $\xi$ and $\varphi_k\in L^1(\mathbb R);~k=0,1,2$
such that
\[\chi_0^p+\hat{\varphi}_2\chi_0^2+\hat{\varphi}_1\chi_0+\hat{\varphi}_0=0\]
on $I_\xi\cap\Pi^1(\Lambda).$ Now, consider a function $f_3\in L^1(\mathbb R)$ such that
$\hat{f_3}(\xi)\neq0$ and $\text{supp }\hat{f_3}\subset I_\xi.$ Let $f_k=f_3\ast\varphi_k;~k=0,1,2.$
Then we can construct a Borel measure $\mu$ which is supported on $\Gamma$ such that
\[\hat\mu(\overline{\xi},\overline{\eta})=\hat{f_0}(\overline{\xi}) +
\chi_0(\overline{\xi})\hat{f_1}(\overline{\xi}) + \chi_0^2(\overline{\xi})\hat{f_2}(\overline{\xi}) +
\chi_0^p(\overline{\xi})\hat{f_3}(\overline{\xi})=0\]
for all $\overline\xi\in I_\xi\cap\Pi^{1^\ast}(\Lambda),$ where $\overline\eta\in\varSigma_{\overline\xi}.$
Since (\ref{exp8}) yields $I_\xi\cap\Pi(\Lambda)=I_\xi\cap\Pi^{1^\ast}(\Lambda),$ it implies that
$\hat\mu\vert_\Lambda=0.$ However, $\mu$ is a non-zero measure which contradicts the fact that
$\left(\Gamma, \Lambda\right)$ is a HUP.

\smallskip

\noindent $\bf{(b).}$  If $\xi\in\Pi(\Lambda)\cap J\subset\Pi^{2^\ast}(\Lambda),$ then
by Lemma \ref{lemma99}, $\xi\in\Pi^p_{2^\ast}(\Lambda).$ Hence there exists an interval
$I_\xi\subset J$ containing $\xi$ and $\varphi_k\in L^1(\mathbb R);~k=0,1,2$
such that
\[\chi_j^p+\hat{\varphi}_2\chi_j^2+\hat{\varphi}_1\chi_j+\hat{\varphi}_0=0\]
on $I_\xi\cap\Pi^2(\Lambda)$ for $j=0,1.$ Let $f_3\in L^1(\mathbb R)$ be such that
$\hat{f_3}(\xi)\neq0$ and $\text{supp }\hat{f_3}\subset I_\xi.$ Denote $f_k=f_3\ast\varphi_k;~k=0,1,2.$
Then we can construct a Borel measure $\mu$ that satisfies
\[\hat\mu(\bar\xi,\bar{\eta_j})=\hat{f_0}(\bar\xi) +
\chi_j(\bar\xi)\hat{f_1}(\bar\xi) + \chi_j^2(\bar\xi)\hat{f_2}(\bar\xi) +
\chi_j^p(\bar\xi)\hat{f_3}(\bar\xi)=0\]
for all $\bar\xi\in I_\xi\cap\Pi^{2^\ast}(\Lambda)$ and $j=0,1.$
Since $I_\xi\cap\Pi(\Lambda)=I_\xi\cap\Pi^{2^\ast}(\Lambda),$ it follows that
$\hat\mu\vert_\Lambda=0,$ though $\mu$ is a non-zero measure.

\smallskip

\noindent $\bf{(c).} $ If $\xi\in\Pi(\Lambda)\cap J\subset\Pi^{3^\ast}(\Lambda),$ then
by Lemma \ref{lemma100}, it follows that $\xi\in\Pi^p_{3^\ast}(\Lambda).$ As $\tau_k\in L^{\Pi^3(\Lambda),\xi}_{loc};~k=0,1,2,$
there exists an interval $I_\xi\subset J$ containing $\xi$ and $\varphi_k\in L^1(\mathbb R)$
such that $\hat{\varphi}_k=\tau_k$ on $I_\xi\cap\Pi^3(\Lambda)$ for $k=0,1,2.$ Let
$f_3\in L^1(\mathbb R)$ be such that $\hat{f_3}(\xi)\neq0$ and $\text{supp }\hat{f_3}\subset I_\xi.$
Denote $f_k=f_3*\varphi_k;~k=0,1,2.$ Since  $X_{\xi}=(\tau_0(\xi),\tau_1(\xi),\tau_2(\xi))$ satisfies
$A^3_{\xi} X_{\xi}=B^p_{\xi},$ we have
\[\tau_0+\chi_j\tau_1+\chi_j^2\tau_2+\chi_j^p=0\] on $I_\xi\cap\Pi^3(\Lambda)$ for $j=0,1,2.$
Hence we can construct a Borel measure $\mu$ such that
\[\hat\mu(\bar\xi,\bar{\eta_j})=\hat{f_0}(\bar\xi)+\chi_j(\bar\xi)\hat{f_1}(\bar\xi)+
\chi_j^2(\bar\xi)\hat{f_2}(\bar\xi)+\chi_j^p(\bar\xi)\hat{f_3}(\bar\xi)=0\]
for all $\bar\xi\in I_\xi\cap\Pi^{3^\ast}(\Lambda)$ and $j=0,1,2.$
As $I_\xi\cap\Pi(\Lambda)=I_\xi\cap\Pi^{3^\ast}(\Lambda),$ we infer that $\hat\mu\vert_\Lambda=0,$
even though $\mu$ is a non-zero measure.
\end{proof}

The next lemma is to deal with the situation that any interval $J\subset I_o$
can not contain only the points of any pair of dispensable sets.
\begin{lemma}\label{lemma9}
There does not exist any interval $J\subset I_o$ such that $\Pi(\Lambda)\cap J$
is contained in $\Pi^{j^\ast}(\Lambda)\cup\Pi^{k^\ast}(\Lambda)~\forall~j\neq k$
and $j,k\in\{1,2,3\}.$
\end{lemma}

\begin{proof}
On the contrary, suppose there exists an interval $J\subset I_o$ such that
$\Pi(\Lambda)\cap J\subset\Pi^{j^\ast}(\Lambda)\cup\Pi^{k^\ast}(\Lambda)~\text{for some}~j\neq k$
and $j,k\in\{1,2,3\}.$ Then we have the following three cases:
\smallskip

\noindent $\textbf{(a).}$ If $\Pi(\Lambda)\cap J\subset\Pi^{1^\ast}(\Lambda)\cup\Pi^{2^\ast}(\Lambda),$
then Equation (\ref{exp8}) yields
\begin{equation}\label{exp18}
J\cap\Pi(\Lambda)=J\cap\left(\Pi^{1^\ast}(\Lambda)\cup\Pi^{2^\ast}(\Lambda)\right).
\end{equation}
We claim that $J\cap\Pi^{2^\ast}(\Lambda)$ is dense in $J.$ If possible, suppose
there exists an interval $I\subset J$ such that $\Pi^{2^\ast}(\Lambda)\cap I=\varnothing.$
Then from (\ref{exp18}), we get $I\cap\Pi(\Lambda)=I\cap\Pi^{1^\ast}(\Lambda)\subset\Pi^{1^\ast}(\Lambda)$
which contradicts Lemma \ref{lemma8}. By Lemma \ref{lemma6}, there exists
an interval $I'\subset J$ such that $I'\subset\bigcup\limits_{j=2}^4\Pi^j(\Lambda).$
This contradicts the assumption that $I_o$ intersects only the dispensable sets.

\smallskip

\noindent $\textbf{(b).}$ If $\Pi(\Lambda)\cap J\subset\Pi^{1^\ast}(\Lambda)\cup\Pi^{3^\ast}(\Lambda),$
then $J\cap\Pi(\Lambda)=J\cap\left(\Pi^{1^\ast}(\Lambda)\cup\Pi^{3^\ast}(\Lambda)\right).$
As similar to the case $\textbf{(a)},$ $J\cap\Pi^{3^\ast}(\Lambda)$ is also dense in $J.$
Hence by Lemma \ref{lemma7}, there exists an interval $I'\subset J$ such that $I'$
is contained in $\Pi^{3^\ast}(\Lambda)\cup\Pi^4(\Lambda).$ Thus in view of Lemma
\ref{lemma8}, we have arrived at a contradiction to the assumption that $I_o$
intersects only the dispensable sets.

\smallskip

\noindent $\textbf{(c).}$ If $\Pi(\Lambda)\cap J\subset\Pi^{2^\ast}(\Lambda)\cup\Pi^{3^\ast}(\Lambda),$
then $J\cap\Pi(\Lambda)=J\cap(\Pi^{2^\ast}(\Lambda)\cup\Pi^{3^\ast}(\Lambda)).$
Hence it follows that $J\cap\Pi^{3^\ast}(\Lambda)$ is dense in $J.$ By using Lemma \ref{lemma7},
there exists an interval $I'\subset J$ such that $I'$ is contained in $\Pi^{3^\ast}(\Lambda)\cup\Pi^4(\Lambda),$
which contradict the assumption that $I_o$ intersects only the dispensable sets.
\end{proof}

Finally, since $\Pi(\Lambda)$ is a dense subset of $\mathbb R,$ in view of Lemmas \ref{lemma8} and
\ref{lemma9}, the only possibility that any interval $J\subset I_o$ would intersect all
the dispensable sets $\Pi^{j^\ast}(\Lambda);~j=1,2,3.$ We claim that $\Pi^{2^\ast}(\Lambda)\cap I_o$
is dense in $I_o.$ Otherwise, there exists an interval $I\subset I_o$ such that $\Pi^{2^\ast}(\Lambda)\cap I=\varnothing.$
Then from (\ref{exp8}), we get $I\cap\Pi(\Lambda)\subset\left(\Pi^{1^\ast}(\Lambda)\cup\Pi^{3^\ast}(\Lambda)\right)$
which contradicts Lemma \ref{lemma9}. Hence by Lemma \ref{lemma6}, there exists an interval
$I'\subset I_o$ such that $I'$ is contained in $\Pi^{2}(\Lambda)\cup\Pi^3(\Lambda)\cup\Pi^4(\Lambda)$
which contradicts the assumption that $I_o$ intersects only the dispensable sets.

\bigskip

\noindent{\bf Concluding remarks:}\\

\noindent $\textbf{(a).}$ We observe a phenomenon of interlacing of
three totally disconnected disjoint dispensable sets ${\Pi^{{(3-j)}^\ast}(\Lambda)}:~j=0,1,2$
which are essentially derived from zero sets of four trigonometric
polynomials.

\smallskip

\noindent $\textbf{(b).}$  If the measure in question is supported on
an arbitrary number of parallel lines, then the size of the dispensable
sets would be larger. Indeed, the method used for the proof of Theorem
\ref{th17} would be highly implicit for a large number of parallel lines.
Since the dispensable sets are totally disconnected, it would be an
interesting question to analyze Heisenberg uniqueness pairs corresponding
to the finite number of parallel lines in terms of Hausdorff dimension of
the dispensable sets.
\smallskip

\noindent $\textbf{(c).}$ If we consider countably many parallel lines, then
whether the projection $\Pi(\Lambda)$ would be still relevant after deleting
the countably many dispensable sets, seems to be a reasonable question. We
leave these questions open for the time being.
\smallskip

\noindent $\textbf{(d).}$
For $p=3,$  in Lemma \ref{lemma7} we have used the fact that any symmetric polynomial in
$a,b,c$ can be expressed as a polynomial in $\tau_j;~j=0,1,2.$ This enables us to define
a function $\rho\in L_{loc}^{\Pi^3(\Lambda),\bar\xi},$ which is crucial in the proof of
Lemma \ref{lemma7}. However, for $p\geq4,$ the functions $\tau_j;~j=0,1,2$ appeared in
Equations (\ref{exp19}) are away from the elementary symmetric polynomials. If we could
identify the space of symmetric polynomials generated by $\tau_j;~j=0,1,2,$ then we can
think to modify the Lemma \ref{lemma7} in terms of $\Pi_{3^\ast}^p(\Lambda)$ that would
help in minimizing the size of the set $\widetilde\Pi(\Lambda).$ Hence a characterization
of $\Lambda$ for four lines problem might be obtained that would be closed to three lines
result. However, an exact analogue of three lines result for a large number of lines is
still open.

\bigskip

\noindent{\bf Acknowledgements:}\\

The authors would like to thank the referee for his/her fruitful suggestion.
The authors wish to thank E. K. Narayanan and Rama Rawat for several fruitful
discussions during preparation of this manuscript. The authors would also like
to gratefully acknowledge the support provided by IIT Guwahati, Government of India.

\bigskip

%\small

\end{document}